\documentclass{amsart}
\usepackage{amsmath}
\usepackage{amssymb}
\usepackage{cite}
\newtheorem{theorem}{Theorem}[section]
\newtheorem{lemma}[theorem]{Lemma}
\newtheorem{corollary}[theorem]{Corollary}
\theoremstyle{definition}

\theoremstyle{remark}

\numberwithin{equation}{section}

\begin{document}

\title{Eigenvalues of Laplace operators on non-bipartite graphs}


\author{$Hongjun~Wang^{1}$}
\address{}
\curraddr{}
\email{}
\thanks{Hongjun~Wang: Department of Mathematics, Hebei University of technology, Tianjin 300401, People's Republic of China. email: 18135515783@163.com}

\author{$Hongmei~Song^{2}$}
\address{}
\curraddr{}
\email{}
\thanks{Hongmei~Song: Department of Mathematics, Hebei University of technology, Tianjin 300401, People's Republic of China. email: shm040356@163.com}

\author{$Jia~Zhao^{3,*}$}
\address{}
\curraddr{}
\email{}
\thanks{*Corresponding author: Jia Zhao, Department of Mathematics, Hebei University of technology, Tianjin 300401, People's Republic of China. email: zhaojia@hebut.edu.cn}


\date{}

\dedicatory{}

\begin{abstract}
This paper considers the comparison between the eigenvalues of Laplace operators with the standard conditions and the anti-standard conditions on non-bipartite graphs which are equilateral or inequilateral. First of all, we show the calculation of the eigenvalues of Laplace operators on equilateral metric graphs with arbitrary edge length. Based on this method, we use the properties of the cosine function and the arccosine function to find the comparison between the eigenvalues of Laplace operators with the standard conditions and the anti-standard conditions on equilateral non-bipartite graphs. In addition, we give the inequalities between standard and anti-standard eigenvalues on a special inequilateral non-bipartite graph.
\end{abstract}

\maketitle
\section{Introduction}
\label{Intro}

Differential operators on metric graphs are a class of unbounded linear operators which are widely used in physics, chemistry, and engineering, see \cite{1,2,3}. The spectral theory of differential operators on metric graphs has become an important part of the spectral theory of differential operators in recent decades, for example \cite{4}.~In this paper, we focus on the Laplace operators on metric graphs, i.e., the second derivative operators on each edge of the graphs.~The most common vertex conditions for the Laplace operators on metric graphs are standard conditions (or Kirchhoff conditions) and anti-standard conditions (or anti-Kirchhoff conditions), see Section 2 below for more details.~We denote Laplace operators subjected to standard conditions and anti-standard conditions on metric graph $\mathrm{\Gamma}$ by $L^{st}(\mathrm{\Gamma})$ and $L^{a/st}(\mathrm{\Gamma})$.

Our main aim is to provide the comparison between the eigenvalues of $L^{st}(\mathrm{\Gamma})$ and $L^{a/st}(\mathrm{\Gamma})$ on non-bipartite graphs with arbitrary edge length.~There have been some results about the comparison between the eigenvalues of $L^{st}(\mathrm{\Gamma})$ and $L^{a/st}(\mathrm{\Gamma})$ on bipartite graphs which can be found in \cite{7,5,6}.~Inspired by \cite{5} which treats the eigenvalues of $L^{st}(\mathrm{\Gamma})$ and $L^{a/st}(\mathrm{\Gamma})$ on equilateral metric graphs with edge length 1, we give the calculation of the eigenvalues of $L^{st}(\mathrm{\Gamma})$ and $L^{a/st}(\mathrm{\Gamma})$ on an equilateral metric graph $\mathrm{\Gamma}$ with arbitrary edge length.~In \cite{6}, the authors show that the positive eigenvalues of Laplace operators with standard conditions and anti-standard conditions on bipartite graphs are equal, and the equation related to the number of even cycles was given. Next, we assume that graph $\mathrm{\Gamma}$ is an equilateral non-bipartite graph, then provide necessary and sufficient conditions for the eigenvalues of $L^{st}(\mathrm{\Gamma})$ and $L^{a/st}(\mathrm{\Gamma})$ to satisfy the following inequality:
 \begin{equation}\label{1.1}
\begin{aligned}
\lambda_{k+1}(L^{st}(\mathrm{\Gamma}))\geq \lambda_{N-n+k}(L^{a/st}(\mathrm{\Gamma})),\quad k\in \mathbb{N},
\end{aligned}
\end{equation}
where $N$ denotes the number of edges, $n$ denotes the number of vertices.

The paper is organized as follows. We introduce two main vertex conditions for Laplace operators in Section 2 and give the calculation of the eigenvalues of $L^{st}(\mathrm{\Gamma})$ and $L^{a/st}(\mathrm{\Gamma})$ on an equilateral metric graph $\mathrm{\Gamma}$ with arbitrary edge length in Section 3. Based on the calculation of the eigenvalues of $L^{st}(\mathrm{\Gamma})$ and $L^{a/st}(\mathrm{\Gamma})$, we give the comparison between the eigenvalues of $L^{st}(\mathrm{\Gamma})$ and $L^{a/st}(\mathrm{\Gamma})$ on a non-bipartite graph $\mathrm{\Gamma}$ with arbitrary edge length by using the properties of the cosine function and the arccosine function in Section 4. We show how the eigenvalues of $L^{a/st}(\mathrm{\Gamma})$ change when vertices of an equilateral odd cycle are increased in corollary \ref{4.1}. The comparison between the eigenvalues of $L^{st}(\mathrm{\Gamma})$ and $L^{a/st}(\mathrm{\Gamma})$ on equilateral non-bipartite graphs is described in the previous part. As a supplement, we give the inequalities between standard and anti-standard eigenvalues on a special inequilateral non-bipartite graph in the final Section. We find that there are similar results on a special inequilateral non-bipartite graph as on equilateral non-bipartite graphs.
\section{Preliminaries}

A metric graph $\mathrm{\Gamma}$ in the paper consists of a finite set of vertices $\mathcal{V}=\{v_{1},\cdots,v_{n}\}$ and a finite set $\mathcal{E}=\{e_{1},\cdots,e_{N}\}$ of edges connecting the vertices, where $n, N$ denote the number of vertices and edges respectively. Moreover, the graph $\mathrm{\Gamma}$ is connected, and each edge $e_{i}$ is assigned a positive length $l_{e_{i}} \in(0,\infty )$. Finally, the graph $\mathrm{\Gamma}$ is also assumed to be simple, i.e., the graph $\mathrm{\Gamma}$ contains no loops, and at most one edge can join two vertices in $\mathrm{\Gamma}$. We denote the total length of a graph by $L(\mathrm{\Gamma})=\sum\limits l_{e_{i}}$ and say that the graph $\mathrm{\Gamma}$ is equilateral if $l_{e_{i}}=l_{e_{j}}$ for any $e_{i}, e_{j}\in \mathcal{E}$. Any edge $e_{i}\in\mathcal{E}$ will be identified with interval $[0,l_{e_{i}}]$, then we write that $o(e_{i})$ and $t(e_{i})$ denote the starting vertex and ending vertex of edge $e_{i}$ respectively.

Let $\mathcal{V}_{1},\mathcal{V}_{2}\subset \mathcal{V}$ be two disjoint sets and satisfy $\mathcal{V}_{1}\cup \mathcal{V}_{2}=\mathcal{V}$. If $e_{i}\cap \mathcal{V}_{1}\neq\emptyset, e_{i}\cap \mathcal{V}_{2}\neq\emptyset$ for any edge $e_{i}\in \mathcal{E}$, we say that $\mathrm{\Gamma}$ is bipartite.

Two vertices $v_{i}$ and $v_{j}$ will be called adjacent (denoted $v_{i}\sim v_{j}$) if there is an edge $e_{k}$ connecting $v_{i}$ and $v_{j}$.
 We give the adjacency matrix $\mathcal{A}:=(a_{ij})_{n\times n}$ of $\mathrm{\Gamma}$ by
\begin{equation}
a_{ij}=\left\{
\begin{aligned}
&1, \qquad if~v_{i}\sim v_{j},\\&0,\qquad if~v_{i}\nsim v_{j},
\end{aligned}
\right.
\end{equation}
then $\mathcal{A}$ is symmetric. We define the transition matrix~$\mathcal{Z}:=(z_{ij})_{n\times n}$ by
\begin{center}
$\mathcal{Z}:=\mathrm{Diag}(\mathcal{A}e)^{-1}\mathcal{A}.$
\end{center}
where $e:=(1)_{n\times 1}$. The signed incidence matrix $\mathcal{D}:=(d_{ij})_{n\times N}$ of $\mathrm{\Gamma}$ is given by
\begin{equation}
d_{ij}=\left\{
\begin{aligned}
-1,&\qquad if~v_{i}=o(e_{j}),\\1,&\qquad if~v_{i}=t(e_{j}),\\0,&\qquad else.
\end{aligned}
\right.
\end{equation}
We give the index mapping $s(i,j)$: $I_{\mathcal{V}}\times I_{\mathcal{V}}\rightarrow I_{\mathcal{E}}$ as follows
\begin{equation}
s(i,j)=\left\{
\begin{aligned}
k,&\qquad if~e_{k}=\{v_{i},v_{j}\},\\1,&\qquad else.
\end{aligned}
\right.
\end{equation}
where~$I_{\mathcal{V}}:=\{1,...,n\}$ and $I_{\mathcal{E}}:=\{1,...,N\}$.~Let $L^{2}(\mathrm{\Gamma})$ denote the Hilbert space $\oplus_{i}L^{2}[0,l_{e_{i}}]$ with the inner product
 $$(f,g)_{L^{2}(\mathrm{\Gamma})}=\sum\limits_{i}\displaystyle\int_{0}^{l_{e_{i}}}f_{i}(x)\overline{g_{i}(x)}\mathrm{d}x,$$
where $f_{i}(x)$ denotes the restriction $f(x)|_{e_{i}}$.~We denote the maximal Laplace operator $\mathcal{L}_{max}$ as follows
\begin{equation}
\left\{
\begin{aligned}
 & \mathcal{L}_{max}f=-f''(x),~f\in \mathrm{Dom}(\mathcal{L}_{max}),\\&\mathrm{Dom}(\mathcal{L}_{max})=\{f\in L^{2}(\mathrm{\Gamma}):~f_{i},~f'_{i}\in AC[0,l_{e_{i}}],\\ &\qquad\qquad\qquad \qquad~~ for~ each~ i\in I_{\mathcal{E}},~\mathcal{L}_{max}f\in L^{2}(\mathrm{\Gamma})\},
\end{aligned}
\right.
\end{equation}
where $AC[0,l_{e_{i}}]$ represents the set of all absolutely real-valued continuous functions on interval $[0,l_{e_{i}}]$.~The standard Laplace operator $L^{st}(\mathrm{\Gamma})$ (also called Kirchhoff Laplace operator) is defined by
\begin{equation}
\left\{
\begin{aligned}
 & L^{st}(\mathrm{\Gamma})f=-f''(x),~f\in \mathrm{Dom}(L^{st}(\mathrm{\Gamma})),\\
 &\mathrm{Dom}(L^{st}(\mathrm{\Gamma}))=\{f\in \mathrm{Dom}(\mathcal{L}_{max}):~f_{1}(v)=f_{2}(v)=\cdots=f_{d_{v}}(v),\\&\qquad\qquad \qquad \quad f'_{1}(v)+f'_{2}(v)+\cdots +f'_{d_{v}}(v)=0\},
\end{aligned}
\right.
\end{equation}
where $d_{v}$ represents the degree of the vertex $v$, and $f'_{i}(v)$ is the derivative at the vertex $v$ taken along the edge $e_{i}$ in the outgoing direction.~The anti-standard Laplace operator $L^{a/st}(\mathrm{\Gamma})$ (also called anti-Kirchhoff Laplace operator) is defined by
\begin{equation}
\left\{
\begin{aligned}
 & L^{a/st}(\mathrm{\Gamma})f=-f''(x),~f\in \mathrm{Dom}(L^{a/st}(\mathrm{\Gamma})),\\
 &\mathrm{Dom}(L^{a/st}(\mathrm{\Gamma}))=\{f\in \mathrm{Dom}(\mathcal{L}_{max}):~f'_{1}(v)=f'_{2}(v)=\cdots=f'_{d_{v}}(v),\\&\qquad\qquad \qquad \qquad
 f_{1}(v)+f_{2}(v)+\cdots +f_{d_{v}}(v)=0\}.
\end{aligned}
\right.
\end{equation}
The spectra of $L^{st}(\mathrm{\Gamma})$ and $L^{a/st}(\mathrm{\Gamma})$ consist of the isolated,
nonnegative eigenvalues with finite multiplicities.
\section{Standard and anti-standard eigenvalues on equilateral metric graphs}

We assume that the graph $\mathrm{\Gamma}$ is an equilateral metric graph.~For each
directed edge $e_{i}$ of~$\mathrm{\Gamma}$, let $[0,l_{e}]$ be a real interval of length $l_{e}$.~We now find the real-valued function $u\in \mathrm{Dom}(\mathcal{L}_{max})$ that satisfies the standard conditions or satisfies the anti-standard conditions. We define the boundary value function matrix $U(x)$ as follows
\begin{center}
  $U(x):=(u_{ij}(x))_{n\times n},\qquad ~u_{ij}(x)=a_{ij}u_{s(i,j)}(\dfrac{l_{e}+l_{e}d_{i,s(i,j)}}{2}-xd_{i,s(i,j)}),$
\end{center}
where $u_{i}(x)$ denotes the restriction $u(x)|_{e_{i}}$.
Hence the equation
$$-u_{i}''(x)=\lambda u_{i}(x),~\forall~i\in\{1,\cdots,N\}$$ can be translated into
\begin{equation}\label{2.2}
-U''(x)=\lambda U(x).
\end{equation}
Then we have
$$U^{*}(x)=U(l_{e}-x),\qquad \qquad U'(x)^{*}=-U'(l_{e}-x).$$
We define the two matrices
\begin{center}
$\Phi:=U(0),\qquad \Psi:=U'(0).$
\end{center}

We consider the standard eigenvalue problem
\begin{equation}\label{2.1}
\left\{
\begin{aligned}
&-u_{j}''(x)=\lambda u_{j}(x),~\forall~j\in\{1,\cdots,N\},\\&u_{j}(v_{i})=u_{k}(v_{i}),~\forall~v_{i}\in e_{j}\cap e_{k},\\&\sum_{j=1}^{N}d_{ij}u_{j}'(v_{i})=0,~\forall~i\in\{1,\cdots,n\},
\end{aligned}
\right.
\end{equation}
then they could be represented by $U(x)$ as follows
\begin{equation}\label{3.1}
\left\{
\begin{aligned}
&-U''(x)=\lambda U(x),\\& U(0)=\phi e^{*}\cdot \mathcal{A},\\& U'(0)e=0,
\end{aligned}
\right.
\end{equation}
where $\phi:=(u(v_{i}))_{n\times1}.$
We consider the anti-standard eigenvalue problem
\begin{equation}\label{2.3}
\left\{\begin{aligned}
&-u_{j}''(x)=\lambda u_{j}(x),~\forall~j\in\{1,\cdots,N\},\\&d_{ij}u'_{j}(v_{i})=d_{ik}u'_{k}(v_{i}),~\forall~v_{i}\in e_{j}\cap e_{k},\\& \sum_{j=1}^{N}d_{ij}^{2}u_{j}(v_{i})=0,~\forall~i\in\{1,\cdots,n\},
\end{aligned}
\right.
\end{equation}
then they could be represented by $U(x)$ as follows
\begin{equation}\label{3.2}
\left\{
\begin{aligned}
&-U''(x)=\lambda U(x),\\&U'(0)=\psi e^{*}\cdot \mathcal{A},\\& U(0)e=0,
\end{aligned}
\right.
\end{equation}
where $\psi:=(u'(v_{i}))_{n\times1}.$
Let $U(x)$ be a nontrivial solution of \eqref{2.2} corresponding to the
eigenvalue $\lambda$. Then $\lambda\in[0, \infty)$ and the form of $U(x)$ is
 \begin{equation}\label{3.3}
U(x)=\left\{
\begin{aligned}
&\Phi+\dfrac{\Phi^{*}-\Phi}{l_{e}}x,~&~\lambda=0,\\&\cos(\sqrt{\lambda }x)\Phi+\dfrac{\sin(\sqrt{\lambda }x)}{\sqrt{\lambda }}\Psi,~&~\lambda>0.
\end{aligned}
\right.
\end{equation}

Let $P:=(p_{ij})_{n\times m}$, $Q:=(q_{ij})_{n\times m}$ and $R:=(r_{ij})_{n\times m}$ be three real matrices. We define the Hadamard product as follows
\begin{center}
  $P\cdot Q=(p_{ij}q_{ij})_{n\times m},$
\end{center} then
\begin{center}
  $P\cdot Q=Q\cdot P,$\quad
$(P\cdot Q)^{*}= P^{*}\cdot Q^{*},$\quad
$(P+Q)\cdot R=P \cdot R+Q \cdot R.$
\end{center}
Moreover, if~$x:=(x_{i})_{n\times1}$~we have
\begin{center}
 $ ((P\cdot Q)x)_{i}=(P\mathrm{Diag}(x_{i})Q^{*})_{ii}.$
\end{center}
Let $\mathrm{\Gamma}$ be a metric graph with
adjacency matrix $\mathcal{A}:=(a_{ij})_{n\times n}$. We define the three matrix spaces
$$\mathcal{M}(\mathrm{\Gamma}):=\{M=(m_{ij})_{n\times n}:~a_{ij}=0\Rightarrow m_{ij}=0\},~$$
$$\quad\mathcal{M}^{-}(\mathrm{\Gamma}):=\{M\in \mathcal{M}(\mathrm{\Gamma}):~M^{*}=-M~and~Me=0\},$$
$$\mathcal{M}^{+}(\mathrm{\Gamma}):=\{M\in \mathcal{M}(\mathrm{\Gamma}):~M^{*}=M~and~Me=0\}.$$
We need to know the multiplicities of the eigenvalues of $L^{st}(\mathrm{\Gamma})$ and $L^{a/st}(\mathrm{\Gamma})$ by the dimensions of the space $\mathcal{M}^{-}(\mathrm{\Gamma})$~and~$\mathcal{M}^{+}(\mathrm{\Gamma})$which can be found in \cite{8}.~The dimensions of $\mathcal{M}^{-}(\mathrm{\Gamma})$~and~$\mathcal{M}^{+}(\mathrm{\Gamma})$ are
\begin{equation}
\begin{aligned}
\mathrm{dim}(\mathcal{M}^{-})=N-n+1,
\end{aligned}
\end{equation}
\begin{equation}\mathrm{dim}(\mathcal{M}^{+})= \begin{cases}N-n+1&\quad if~\mathrm{\Gamma}~is~bipartite, \\N-n&\quad if~\mathrm{\Gamma}~is~non\mathrm{-}bipartite.
 \end{cases}\end{equation}
Let $D$ be a $n\times n$ real matrix. We call that $D$ is reducible if there exists a $n\times n$ permutation matrix $P$ such that
$$PDP^{\top}=\begin{bmatrix}
D_{11}&0\\D_{21}&D_{22}
\end{bmatrix},$$
where $D_{11}$ is $r\times r$ matrix and $D_{21}$ is $(n-r)\times r$ matrix and $D_{22}$ is $(n-r)\times (n-r)$ matrix. We say that $D$ is irreducible if there is no such a matrix $P$.

We want to obtain the eigenvalues of $L^{st}(\mathrm{\Gamma})$ and $L^{a/st}(\mathrm{\Gamma})$ on equilateral metric graphs by the eigenvalues of the transition matrix $\mathcal{Z}$. By using the theory in \cite[Lemma 2.36]{5}, we can give the following lemma.
\begin{lemma}\label{h1}
 Let graph $\mathrm{\Gamma}$ be a metric graph with the transition matrix $\mathcal{Z}$.~Then
the transition matrix $\mathcal{Z}$ is irreducible and $\sum_{j=1}^{n}z_{ij}=1$ for all $i\in\{1,\cdots n\}$. Moreover $\mathcal{Z}$ has the real eigenvalues~$\mu_{1},\mu_{2},\cdots,\mu_{n}$ which can be sorted. According to their multiplicities, we have
$$1=\mu_{1}>\mu_{2}\geq\cdots\geq\mu_{n}\geq-1.$$
Furthermore, if $\mu_{n}=-1$ then $\mathrm{\Gamma}$ is bipartite.
\end{lemma}
\begin{theorem}\label{t1}
 Let $\mathrm{\Gamma}$ be an equilateral metric graph with the transition matrix $\mathcal{Z}$. Then the spectrum of $L^{st}(\mathrm{\Gamma})$ is given by
\begin{center}
  $\sigma(L^{st}(\mathrm{\Gamma}))=\{0\} \cup\{\lambda>0:\cos(\sqrt{\lambda}l_{e})\in\sigma(\mathcal{Z})\}\cup\{\lambda>0:\cos(\sqrt{\lambda}l_{e})=-1\}.$
\end{center}
The multiplicities of the eigenvalues are given by
\begin{equation}\notag
m(\lambda)=\left\{
\begin{aligned}
&1, \qquad\qquad\qquad\qquad\qquad\qquad~if~\lambda=0,
\\&\mathrm{dimker}(\mathcal{Z}-\cos(\sqrt{\lambda}l_{e})I),~if~\sin(\sqrt{\lambda}l_{e})\neq0,\\
&N-n+2, \qquad\qquad\qquad\qquad~if~\cos(\sqrt{\lambda}l_{e})=1~and ~\lambda>0,
\\&N-n+2, \qquad\qquad\qquad\qquad~if~\cos(\sqrt{\lambda}l_{e})=-1~ and ~\mathrm{\Gamma} ~is~ bipartite,\\
&N-n ,\qquad \qquad\qquad\qquad if~\cos(\sqrt{\lambda}l_{e})=-1~ and ~\mathrm{\Gamma} ~is~non\mathrm{-}bipartite,
\end{aligned}
\right.
\end{equation}
where $I$ is an identity matrix.
\end{theorem}
\begin{proof}
We consider the first case $\lambda=0$. From \eqref{3.3} we know that all solutions of \eqref{2.2} corresponding to $\lambda=0$, and the respective derivatives are obtained by
$$U(x)=\Phi+\dfrac{\Phi^{*}-\Phi}{l_{e}}x,\qquad U'(x)=\dfrac{\Phi^{*}-\Phi}{l_{e}}.$$
By using $\Phi=U(0)=\phi e^{*}\cdot \mathcal{A}$, we get a new expression
 $$U'(0)=\dfrac{\Phi^{*}-\Phi}{l_{e}} =\dfrac{(e\phi^{*}-\phi e^{*})\cdot \mathcal{A}}{l_{e}}.$$
 Since $U'(0)e=0$, then we have
 \begin{align*}
 &(\mathcal{A}\cdot e \phi^{*})e=(\mathcal{A}\cdot\phi e^{*})e \\
 \Rightarrow&(\mathcal{A}\phi e ^{*})_{ii}=(\mathcal{A}e\phi^{*})_{ii}\\
 \Leftrightarrow &\mathcal{A}\phi =\mathrm{Diag}(\mathcal{A}e)\phi\\
 \Leftrightarrow& \mathrm{Diag}(\mathcal{A}e)^{-1}\mathcal{A}\phi=\phi\\
 \Leftrightarrow &\mathcal{Z}\phi=\phi.
\end{align*}
Thus $\phi$ is an eigenvector of $\mathcal{Z}$ belonging to the eigenvalue $1$ with the multiplicity $1$, then the multiplicity of the eigenvalue $\lambda=0$ is 1.

We consider the second case $\lambda>0$. From \eqref{3.3} we see that all solutions of \eqref{2.2}, and the respective derivatives are obtained by
$$U(x)=\cos(\sqrt{\lambda }x)\Phi+\dfrac{\sin(\sqrt{\lambda }x)}{\sqrt{\lambda }}\Psi,~~U'(x)=-\sqrt{\lambda }\sin(\sqrt{\lambda }x)\Phi+\cos(\sqrt{\lambda }x)\Psi.
 $$

We assume that $\sin(\sqrt{\lambda}l_{e})\neq0$. Since $U(l_{e})=U^{*}(0)=\Phi^{*}$ and $\Phi=U(0)=\phi e^{*}\cdot \mathcal{A}$, we obtain that
\begin{align*}
\Psi&=\dfrac{\sqrt{\lambda }}{\sin(\sqrt{\lambda }l_{e})}(\Phi^{*}-\Phi\cos(\sqrt{\lambda }l_{e}))\\
& =\dfrac{\sqrt{\lambda }}{\sin(\sqrt{\lambda }l_{e})}(e\phi^{*}-\cos(\sqrt{\lambda }l_{e})\phi e^{*})\cdot \mathcal{A}.
\end{align*}
Due to $U'(0)e=0$ and $\Psi=U'(0)$, we have
\begin{align*}
&(e \phi^{*}\cdot \mathcal{A} )e=(\cos(\sqrt{\lambda }l_{e})\phi e^{*}\cdot \mathcal{A})e \\
 \Rightarrow&(\mathcal{A}\phi e ^{*})_{ii}=\cos(\sqrt{\lambda }l_{e})(\mathcal{A}e\phi^{*})_{ii}\\
 \Leftrightarrow &\mathcal{A}\phi =\cos(\sqrt{\lambda }l_{e})\mathrm{Diag}(\mathcal{A}e)\phi\\
 \Leftrightarrow&\mathcal{Z}\phi=\cos(\sqrt{\lambda }l_{e})\phi.
\end{align*}
Then multiplicity of the eigenvalue $\lambda$ is dimker$(\mathcal{Z}-\cos(\sqrt{\lambda}l_{e})I)$ when $\sin(\sqrt{\lambda}l_{e})\neq0$.

For the remaining case of $\sin(\sqrt{\lambda}l_{e})=0$, we will distinguish the cases $\cos(\sqrt{\lambda }l_{e})=1$ and $\cos(\sqrt{\lambda }l_{e})=-1$.

Case 1: We will begin with the $\cos(\sqrt{\lambda }l_{e})=1$. Since every solution $\Phi^{*}=\Phi=\phi e^{*}\cdot \mathcal{A}$, the solution space with vanishing $\Psi$ is one dimension. On the other hand, we have $\Psi^{*}=-\Psi$ by using $U'(l_{e})=-U'(0)^{*}=-\Psi^{*}$. Then according to $\Psi e=U'(0)e=0$, the solution space with vanishing $\Phi$ of \eqref{3.1} is isomorphic to $\mathcal{M}^{-}$ with $\dim(\mathcal{M}^{-})=N-n+1$. Hence the multiplicity of $\lambda$ is $N-n+2$.

Case 2: When $\cos(\sqrt{\lambda }l_{e})=-1$, we have $\Psi=\Psi^{*}$ by using $U'(l_{e})=-U'(0)^{*}=-\Psi^{*}$. Then according to $\Psi e=U'(0)e=0$, the solution space with vanishing $\Phi$ of \eqref{3.1} is isomorphic to $\mathcal{M}^{+}$ with $\dim(\mathcal{M}^{+})$ which is either $N-n+1$ or $N-n$ depending on whether $\mathrm{\Gamma}$ is bipartite or not. On the other hand, since $\Phi=\phi e^{*}\cdot \mathcal{A}$, the solution space with vanishing $\Psi$ is one dimension. But since $U(l_{e})=U^{*}(0)=\Phi^{*}$, we have $\Phi=-\Phi^{*}$ which is possible for non-trivial $\Phi$ if $\mathrm{\Gamma}$ is bipartite, we cannot obtain $\Phi=-\Phi^{*}$ if $\mathrm{\Gamma}$ is  non-bipartite.~Hence the multiplicity of $\lambda$ is $N-n+2$ if $\mathrm{\Gamma}$ is bipartite, the multiplicity of $\lambda$ is $N-n$ if $\mathrm{\Gamma}$ is non-bipartite. We have proven the theorem.
\end{proof}
\begin{theorem}\label{t1}
Let graph $\mathrm{\Gamma}$ be an equilateral metric graph with the transition matrix $\mathcal{Z}$. Then spectrum of $L^{a/st}(\mathrm{\Gamma})$ is given by
  \begin{center}
  $\sigma(L^{a/st}(\mathrm{\Gamma}))=\{0\}\cup\{\lambda>0:-\cos(l_{e}\sqrt{\lambda})\in\sigma(\mathcal{Z})\}\cup\{\lambda>0:\cos(l_{e}\sqrt{\lambda})=1\}.$
\end{center} The multiplicities of the eigenvalues are given by
\begin{equation}\notag
m(\lambda)=\left\{
\begin{aligned}
&N-n+1,\qquad\qquad \qquad\qquad ~~if~\lambda=0~and~\mathrm{\Gamma}~is~ bipartite,\\
&N-n,\qquad\qquad\qquad\qquad if~\lambda=0~and~\mathrm{\Gamma}~is~non\mathrm{-}bipartite,\\
&\mathrm{dimker}(\mathcal{Z}+\cos(\sqrt{\lambda}l_{e})I),~if~\sin(\sqrt{\lambda}l_{e})\neq0,\\
&N-n+2, \qquad\qquad\qquad~\qquad if~\cos(\sqrt{\lambda}l_{e})=-1 ,
\\&N-n+2,\qquad\qquad\qquad\qquad~if~\cos(\sqrt{\lambda}l_{e})=1~and ~\mathrm{\Gamma}~is~ bipartite,\\
&N-n , \qquad\qquad\qquad\qquad if~\cos(\sqrt{\lambda}l_{e})=1~ and ~\mathrm{\Gamma}~ is ~non\mathrm{-}bipartite.
\end{aligned}
\right.
\end{equation}
\end{theorem}
\begin{proof}
We consider the first case $\lambda=0$. From \eqref{3.3} we know that all solutions of \eqref{2.2} with
$\lambda=0$, and the respective derivatives are obtained by
$$U(x)=\Phi+\dfrac{\Phi^{*}-\Phi}{l_{e}}x,~~U'(x)=\dfrac{\Phi^{*}-\Phi}{l_{e}}.$$
 Using integration by parts, we obtain
 \begin{equation}\label{3.4}
\begin{aligned}
0=\sum\limits_{j=1}^{N}\int_{0}^{l_{e}}u_{j}''u_{j}\mathrm{d}x_{j}=\sum\limits_{j=1}^{N}[u_{j}'u_{j}]_{0}^{l_{e}}-\sum\limits_{j=1}^{N}\int_{0}^{l_{e}}u_{j}'^{2}\mathrm{d}x_{j}.
\end{aligned}
\end{equation}
Simplify the first term on the right side of the equation as follows
\begin{align*}
 \sum\limits_{j=1}^{N}[u_{j}'u_{j}]_{0}^{l_{e}}&=\sum\limits_{j=1}^{N}u_{j}'(l_{e})u_{j}(l_{e})
-\sum\limits_{j=1}^{N}u_{j}'(0)u_{j}(0) \\
&=\sum_{j=1}^{N}\sum\limits_{v_{i}\in e_{j}}d_{ij}u_{j}'(v_{i})u_{j}(v_{i})\\
 &=\sum\limits_{i=1}^{n}\sum\limits_{j=1}^{d_{v_{i}}}d_{ij}u_{j}'(v_{i})u_{j}(v_{i}),
\end{align*}
 where $d_{v_{i}}$ denotes the degree of the vertex $v_{i}$. Considering $d_{ij}u_{j}'(v_{i})=d_{ik}u_{k}'(v_{i})$~
at fixed vertex $v_{i}$,~we can write $u'(v_{i}):=d_{ij}u'_{j}(v_{i})$. Therefore we have
\begin{equation}
\begin{aligned}
\sum\limits_{j=1}^{N}[u_{j}'u_{j}]_{0}^{l_{e}}=\sum\limits_{i=1}^{n}\sum\limits_{j=1}^{d_{v_{i}}}d_{ij}u_{j}'(v_{i})u_{j}(v_{i})=\sum\limits_{i=1}^{n}u'(v_{i})\sum\limits_{j=1}^{d_{v_{i}}}u_{j}(v_{i})=0.
\end{aligned}
\end{equation}
Then \eqref{3.4} can be written as
\begin{equation}
\begin{aligned}
0=\sum\limits_{j=1}^{N}\int_{0}^{l_{e}}u_{j}''u_{j}\mathrm{d}x_{j}=
-\sum\limits_{j=1}^{N}\int_{0}^{l_{e}}u_{j}'^{2}\mathrm{d}x_{j}\leq0.
\end{aligned}
\end{equation}
Thus $u'_{j}\equiv0$. We have $\Phi=\Phi^{*}$, and according to $\Phi e=0$, this solution space is isomorphic to $\mathcal{M}^{+}$.
Hence the multiplicity of $\lambda=0$ is $N-n+1$ if $\mathrm{\Gamma}$ is bipartite, the multiplicity of $\lambda=0$ is $N-n$ if $\mathrm{\Gamma}$ is non-bipartite.

We consider the second case $\lambda>0$. From \eqref{3.3} we know that all solutions of \eqref{2.2} and the respective derivatives are obtained by
$$U(x)=\cos(\sqrt{\lambda }x)\Phi+\dfrac{\sin(\sqrt{\lambda }x)}{\sqrt{\lambda }}\Psi,\qquad U'(x)=-\sqrt{\lambda }\sin(\sqrt{\lambda }x)\Phi+\cos(\sqrt{\lambda }x)\Psi.$$

We assume that $\sin(\sqrt{\lambda}l_{e})\neq0$. Using $U'(l_{e})=-U'(0)^{*}=-\Psi^{*}$, we obtain
$$-\Psi^{*}=-\sqrt{\lambda }\sin(\sqrt{\lambda }l_{e})\Phi+\cos(\sqrt{\lambda }l_{e})\Psi$$
$$\Leftrightarrow\Phi=\dfrac{1}{\sqrt{\lambda }\sin(\sqrt{\lambda }l_{e})}(\cos(\sqrt{\lambda }l_{e})\Psi+\Psi^{*}).\qquad $$
Due to $\Psi=U'(0)=\psi e^{*}\cdot \mathcal{A}$ and $\Phi e=U(0)e=0$, we have
$$(\mathcal{A}\cdot e\psi^{*})e=-\cos(\sqrt{\lambda }l_{e})(\mathcal{A}\cdot \psi e^{*})e$$
~$$\qquad\qquad \mathcal{Z}\psi=-\cos(\sqrt{\lambda }l_{e})\psi.\qquad \qquad\qquad\quad$$
Then multiplicity of the eigenvalue $\lambda$ is dimker$(\mathcal{Z}+\cos(\sqrt{\lambda}l_{e})I)$ when $\sin(\sqrt{\lambda}l_{e})\neq0$.

For $\sin(\sqrt{\lambda}l_{e})=0$, we will distinguish
the cases $\cos(\sqrt{\lambda }l_{e})=1$ and $\cos(\sqrt{\lambda }l_{e})=-1$.

Case 1:~We will begin with the $\cos(\sqrt{\lambda }l_{e})=-1$,
since $\Psi^{*}=\Psi=\psi e^{*}\cdot \mathcal{A}$, the solution space with vanishing $\Phi$ of \eqref{3.2} is one dimension. On the other hand, we have $\Phi^{*}=-\Phi$ by using $U(l_{e})=U^{*}(0)=\Phi^{*}$. Then according to $\Phi e=U(0)e=0$, the solution space with vanishing $\Psi$ of \eqref{3.2} is isomorphic to $\mathcal{M}^{-}$ with $\dim(\mathcal{M}^{-})=N-n+1$. Hence the multiplicity of $\lambda$ is $N-n+2$.

Case 2:~When $\cos(\sqrt{\lambda }l_{e})=1$, we have $\Phi^{*}=\Phi$ by using $U(l_{e})=U(0)^{*}=\Phi^{*}$. Then according to $\Phi e=U(0)e=0$, the solution space with vanishing $\Psi$ of \eqref{3.2} is isomorphic to $\mathcal{M}^{+}$ with $\dim(\mathcal{M}^{+})$ which is either $N-n+1$ or $N-n$ depending on whether $\mathrm{\Gamma}$ is bipartite or not. On the other hand, since $\Psi=\psi e^{*}\cdot \mathcal{A}$, the solution space with vanishing $\Phi$ is one dimension. But since $U'(l_{e})=-U'(0)^{*}=-\Psi^{*}$, we see $\Psi=-\Psi^{*}$ which is possible for non-trivial $\Psi$ if $\mathrm{\Gamma}$ is bipartite, we cannot obtain $\Psi=-\Psi^{*}$ if $\mathrm{\Gamma}$ is non-bipartite. Hence the multiplicity of $\lambda$ is $N-n+2$ if $\mathrm{\Gamma}$ is bipartite, the multiplicity of $\lambda$ is $N-n$ if $\mathrm{\Gamma}$ is non-bipartite. We have proven the theorem.
\end{proof}

\section{The relation between standard and anti-standard eigenvalues on equilateral non-bipartite graphs}

The relation between the eigenvalues of $L^{st}(\mathrm{\Gamma})$ and $L^{a/st}(\mathrm{\Gamma})$ on equilateral  bipartite graphs has achieved some good results, please refer to \cite{7} and \cite{6} for details. However, it has remained unnoticed that there is a relation between the eigenvalues of $L^{st}(\mathrm{\Gamma})$ and $L^{a/st}(\mathrm{\Gamma})$ on equilateral non-bipartite graphs.~The comparison between the eigenvalues of Laplace operators with the standard conditions and the anti-standard conditions on equilateral non-bipartite graphs is as follows.

We assume that the graph $\mathrm{\Gamma}$ is an equilateral non-bipartite graph with edge length $l_{e}$.~The number 1 must be the eigenvalue of the corresponding transition matrix $\mathcal{Z}$, and let the eigenvalues of $\mathcal{Z}$ be $\mu_{1},\mu_{2},\cdots,\mu_{n-1},\mu_{n}$,~then we have
 $$1=\mu_{1}>\mu_{2}\geq\cdots\mu_{n-1}\geq\mu_{n}>-1.$$
 Define the~$\mathbb{N}_{0}^{(m)}$~and the~$\mathbb{N}^{(m)}$~as follows
\begin{center}
  ~$\mathbb{N}_{0}^{(m)}=\{\underbrace{0,\cdots, 0}_{m},\underbrace{1,\cdots, 1}_{m},2\cdots\}$,\qquad $\mathbb{N}^{(m)}=\{\underbrace{1,\cdots,1}_{m},\underbrace{2,\cdots, 2}_{m},3\cdots\}.$
\end{center}
The $l_{e}^{2}$ times eigenvalues of $L^{st}(\mathrm{\Gamma})$ and $L^{a/st}(\mathrm{\Gamma})$ are listed in Table~1,
\begin{center}
\begin{tabular}{c|cccc}
 & standard &   anti-standard \\
 \hline
   $\lambda=0$ &  $\{0\}$ &         $\{0\}$~or~$\emptyset$\\
   $ \cos(l_{e}\sqrt{\lambda})=1$ &   $\{(2k\pi)^{2}:k\in \mathbb{N}^{(N-n+2)}\}$ &         $\{(2k\pi)^{2}:k\in \mathbb{N}^{(N-n)}\}$ \\
    $ \sin(l_{e}\sqrt{\lambda})\neq0$ &   $\{(2k\pi\pm\arccos(\mu_{m}))^{2}:$ &
                                       $\{((2k+1)\pi\pm\arccos(\mu_{m}))^{2}:$\\
                                    &     $  ~k\in \mathbb{N}_{0}, m\in\{2,\cdots,n\}\}$ &     $~k\in \mathbb{N}_{0},m\in\{2,\cdots,n\}\}$\\
    $ \cos(l_{e}\sqrt{\lambda})=-1$ &    $ \{(2k+1)^{2}\pi^{2}:k\in \mathbb{N}_{0}^{(N-n)}\}$ &        $ \{(2k+1)^{2}\pi^{2}:k\in \mathbb{N}_{0}^{(N-n+2)}\}$. \\
  \end{tabular}
  $$\textbf{ Table~1~}$$
\end{center}
\begin{theorem}\label{t1}
Let graph $\mathrm{\Gamma}$ be an equilateral non-bipartite graph with the transition matrix $\mathcal{Z}$, then we have
\begin{equation}\label{4.20}
\begin{aligned}
\lambda_{k+1}(L^{st}(\mathrm{\Gamma}))\geq \lambda_{k+N-n}(L^{a/st}(\mathrm{\Gamma})),\quad k\in \mathbb{N}
\end{aligned}
\end{equation}
hold if and only if
\begin{equation}\label{4.11}
\begin{aligned}
\mu_{2}\leq -\mu_{n},\mu_{3}\leq -\mu_{n-1},\cdots,\mu_{n-1}\leq -\mu_{3},\mu_{n}\leq -\mu_{2},\qquad
\end{aligned}
\end{equation}
 \begin{equation}\label{4.2}
\begin{aligned}
\mu_{2}\geq -\mu_{n-2},\mu_{3}\geq -\mu_{n-3},\cdots,\mu_{n-3}\geq -\mu_{3},\mu_{n-2}\geq -\mu_{2},
\end{aligned}
\end{equation}
where $\mu_{1},\mu_{2},\cdots,\mu_{n-1},\mu_{n}$ are the eigenvalues of the $\mathcal{Z}$.
\end{theorem}
\begin{proof}~We can assume that each edge has length $l_{e}$.~Because the arccosine function~$y=\arccos(x)$: [-1,1]$\rightarrow$ $[0,\pi]$, and $y$ is monotonically decreasing in interval [-1,1], the spectra of $L^{st}(\mathrm{\Gamma})$ and $L^{a/st}(\mathrm{\Gamma})$ are
\begin{align*}
\sigma(L^{st}(\mathrm{\Gamma}))=&\Biggr\{0,\left(\frac{\arccos(\mu_{2})}{l_{e}}\right)^{2},\left(\frac{\arccos(\mu_{3})}{l_{e}}\right)^{2},\cdots,\left(\frac{\arccos(\mu_{n})}{l_{e}}\right)^{2},\\
&\underbrace{\left(\frac{\pi}{l_{e}}\right)^{2},\cdots,\left(\frac{\pi}{l_{e}}\right)^{2}}_{N-n},
\left(\frac{2\pi-\arccos(\mu_{n})}{l_{e}}\right)^{2}, \cdots,\left(\frac{2\pi-\arccos(\mu_{3})}{l_{e}}\right)^{2}, \\& \left(\frac{2\pi-\arccos(\mu_{2})}{l_{e}}\right)^{2},
\underbrace{\left(\frac{2\pi}{l_{e}}\right)^{2},\cdots,\left(\frac{2\pi}{l_{e}}\right)^{2}}_{N-n+2},\cdots\Biggr\},\\ &\lambda_{k}\left(L^{st}(\mathrm{\Gamma})\right)\leq \lambda_{k+1}\left(L^{st}(\mathrm{\Gamma})\right),~k\in\mathbb{N}.
\end{align*}
\begin{align*}
\sigma(L^{a/st}(\mathrm{\Gamma}))=&\Biggr\{\underbrace{0,\cdots,0}_{N-n},\left(\dfrac{\pi-\arccos(\mu_{n})}{l_{e}}\right)^{2},\cdots,\left(\dfrac{\pi-\arccos(\mu_{3})}{l_{e}}\right)^{2},
\qquad\qquad\qquad\qquad\qquad\qquad\qquad\\
& \left(\dfrac{\pi-\arccos(\mu_{2})}{l_{e}}\right)^{2}, \underbrace{\left(\dfrac{\pi}{l_{e}}\right)^{2},\cdots,\left(\dfrac{\pi}{l_{e}}\right)^{2}}_{N-n+2},
\left(\dfrac{\pi+\arccos(\mu_{2})}{l_{e}}\right)^{2},\\
&\left(\dfrac{\pi+\arccos(\mu_{n})}{l_{e}}\right)^{2}, \underbrace{\left(\dfrac{2\pi}{l_{e}}\right)^{2},\cdots,\left(\dfrac{2\pi}{l_{e}}\right)^{2}}_{N-n},\cdots \Biggr\},\\
&\lambda_{k}(L^{a/st}\left(\mathrm{\Gamma})\right)\leq \lambda_{k+1}(L^{a/st}\left(\mathrm{\Gamma})\right),~~k\in\mathbb{N}.
\end{align*}

We now consider the comparison between the eigenvalues of $L^{st}(\mathrm{\Gamma})$ and~$L^{a/st}(\mathrm{\Gamma})$ in interval $\left(0,(\frac{2\pi}{l_{e}})^{2}\right]$.~The number of the eigenvalues of $L^{st}(\mathrm{\Gamma})$ and $L^{a/st}(\mathrm{\Gamma})$ is same in interval $\left(0,(\frac{2\pi}{l_{e}})^{2}\right]$.~We consider the following inequality:
 \begin{equation}\label{4.3}
\begin{aligned}
\arccos(\mu_{2})\geq(\pi-\arccos(\mu_{n})).
\end{aligned}
\end{equation}
The function $\cos(x)$ acts on both sides of inequality (\ref{4.3}), and since cosine function $y=\cos(x)$ is monotonically decreasing in interval $(0,\pi]$, we have
$$\mu_{2}\leq -\mu_{n}.$$
For the following inequalities:
$$\arccos(\mu_{3})\geq\left(\pi-\arccos(\mu_{n-1})\right),$$
$$\cdots,$$
$$\arccos(\mu_{n-1})\geq\left(\pi-\arccos(\mu_{3})\right),$$
$$\arccos(\mu_{n})\geq(\pi-\arccos(\mu_{2})),\quad$$
repeating our analysis, we have
$$\mu_{3}\leq -\mu_{n-1},~\cdots,~\mu_{n-1}\leq -\mu_{3},~\mu_{n}\leq -\mu_{2}.$$

Since the arccosine function $y=\arccos(x)$: [-1,1]$\rightarrow$ $[0,\pi]$, then we have
$$(2\pi-\arccos(\mu_{n}))\geq\pi,~(2\pi-\arccos(\mu_{n-1}))\geq\pi.$$

We analyze following inequality:
 \begin{equation}\label{4.4}
\begin{aligned}
(2\pi-\arccos(\mu_{n-2}))\geq(\pi+\arccos(\mu_{2})),
\end{aligned}
\end{equation}
The function $\cos(x)$ acts on both sides of inequality (\ref{4.4}), and since cosine function $y=\cos(x)$ is monotonically increasing in interval $(\pi,2\pi]$, we have
$$\mu_{n-2}\geq-\mu_{2}.$$
For the following inequalities:
$$(2\pi-\arccos(\mu_{n-3}))\geq(\pi+\arccos(\mu_{3})),$$$$~\cdots,~$$$$(2\pi-\arccos(\mu_{3}))\geq(\pi+\arccos(\mu_{n-3})),$$
$$~(2\pi-\arccos(\mu_{2}))\geq(\pi+\arccos(\mu_{n-2})),~$$
then we have
$$\mu_{n-3}\geq -\mu_{3},~\cdots,~\mu_{3}\geq -\mu_{n-3},~\mu_{2}\geq -\mu_{n-2}.$$

For the comparison between the eigenvalues of $L^{st}(\mathrm{\Gamma})$ and~$L^{a/st}(\mathrm{\Gamma})$ in interval $\left((\frac{2k\pi}{l_{e}})^{2},(\frac{2(k+1)\pi}{l_{e}})^{2}\right],~k\in \mathbb{N}$. Because the number of eigenvalues of $L^{st}(\mathrm{\Gamma})$ and $L^{a/st}(\mathrm{\Gamma})$ is same in interval $\left((\frac{2k\pi}{l_{e}})^{2},(\frac{2(k+1)\pi}{l_{e}})^{2}\right],$~$k\in \mathbb{N}$, and according to $\cos(2k\pi+\sqrt{\lambda})=\cos(\sqrt{\lambda}),~\sqrt{\lambda}\in(0,\frac{2\pi}{l_{e}}]$, the comparison between the eigenvalues of $L^{st}(\mathrm{\Gamma})$ and $L^{a/st}(\mathrm{\Gamma})$ in interval $\left((\frac{2k\pi}{l_{e}})^{2},(\frac{2(k+1)\pi}{l_{e}})^{2}\right],~k\in \mathbb{N}$ can get the same results as in interval $\left(0,(\frac{2\pi}{l_{e}})^{2}\right]$. Conversely, based on (\ref{4.11}) and (\ref{4.2}), we get (\ref{4.20}) by using the properties of the arccosine function.
We have proven the theorem.\end{proof}

For a given finite non-bipartite graph, the eigenvalues of the corresponding transition matrix can be easily calculated. Thus we can easily judge whether the eigenvalues of the Laplace operators satisfy (\ref{4.20}) by judging whether the eigenvalues of transition matrix satisfy (\ref{4.11}) and (\ref{4.2}). To provide one of the simplest examples, consider that $\mathrm{\Gamma}$ is a regular pentagon with edge length 1 which is a cycle with 5 edges. In this case, all eigenvalues of $\mathcal{Z}$ are $\mu_{1}=1,\mu_{2}=\mu_{3}=\frac{\sqrt{5}}{4}-\frac{1}{4},\mu_{4}=\mu_{5}=-\frac{\sqrt{5}}{4}-\frac{1}{4}$. The values of arccosine function of $\mu_{2}$, $\mu_{3}$, $\mu_{4}$ and $\mu_{5}$ are not clear. It is not easy for us to directly compare the size between the eigenvalues of $L^{st}(\mathrm{\Gamma})$ and~$L^{a/st}(\mathrm{\Gamma})$, but the eigenvalues of $\mathcal{Z}$ satisfy $\mu_{2}\leq -\mu_{5},\mu_{3}\leq -\mu_{4},\mu_{2}\geq -\mu_{3},\mu_{3}\geq -\mu_{2}$, we have
\begin{equation}\notag
\begin{aligned}
\lambda_{k+1}(L^{st}(\mathrm{\Gamma}))\geq \lambda_{k}(L^{a/st}(\mathrm{\Gamma})),\quad k\in \mathbb{N}.
\end{aligned}
\end{equation}

We next consider the changing rules of the anti-standard eigenvalues when vertices of the metric graphs are increased.
Let $\widetilde{\mathrm{\Gamma}}$ be a metric graph obtained by adding some vertices  to each edge of the metric graph $\mathrm{\Gamma}$, then we have
\begin{center}
 $ \lambda_{k}(L^{st}(\mathrm{\Gamma}))=\lambda_{k}(L^{st}(\widetilde{\mathrm{\Gamma}})),\quad k\in \mathbb{N}.$
\end{center}
\begin{corollary}\label{4.1}
Assume that graph $\mathrm{\Gamma}$ is an odd equilateral cycle, and let $\widetilde{\mathrm{\Gamma}}$ be an equilateral bipartite graph obtained by adding some vertices in the middle of each edge of the graph $\mathrm{\Gamma}$. If the eigenvalues of the transition matrix $\mathcal{Z}$ satisfy \eqref{4.11} and \eqref{4.2}, then we have
$$ \lambda_{k+1}(L^{a/st}(\widetilde{\mathrm{\Gamma}}))\geq\lambda_{k}(L^{a/st}(\mathrm{\Gamma})),\quad k\in \mathbb{N}.$$
\end{corollary}

\begin{proof}
For the graph $\mathrm{\Gamma}$ which is an equilateral non-bipartite graph, if the eigenvalues of the transition matrix $\mathcal{Z}$ satisfy (\ref{4.11}) and (\ref{4.2}), then we have
 $$\lambda_{k+1}(L^{st}(\mathrm{\Gamma}))\geq \lambda_{k}(L^{a/st}(\mathrm{\Gamma})),\quad k\in \mathbb{N}.$$
 After adding vertices, the graph $\widetilde{\mathrm{\Gamma}}$ is an equilateral bipartite graph, we have
 $$\lambda_{k}(L^{st}(\widetilde{\mathrm{\Gamma}}))=\lambda_{k}(L^{a/st}(\widetilde{\mathrm{\Gamma}})),\quad k\in \mathbb{N}.$$
The following relation between the eigenvalues of $L^{st}(\mathrm{\Gamma})$ and $L^{st}(\widetilde{\mathrm{\Gamma}})$
$$\lambda_{k}(L^{st}(\mathrm{\Gamma}))=\lambda_{k}(L^{st}(\widetilde{\mathrm{\Gamma}})),\quad k\in \mathbb{N}.$$Hence
$$ \lambda_{k+1}(L^{a/st}(\widetilde{\mathrm{\Gamma}}))\geq\lambda_{k}(L^{a/st}(\mathrm{\Gamma})),\quad k\in \mathbb{N}.$$
This completes the proof.
\end{proof}
\section{ Inequalities between standard and anti-standard eigenvalues on a special inequilateral non-bipartite graph}

The previous method is only valid for the calculation of the eigenvalues of Laplace operators on equilateral graphs.~Thus in this section, we will present a different approach which solves the eigenvalues of Laplace operators on inequilateral metric graphs.

Since the eigenvalues of $L^{st}(\mathrm{\Gamma})$ and $L^{a/st}(\mathrm{\Gamma})$ are always non-negative real numbers, we have $\lambda=k^{2}.$
We denote $2N\times2N$ global scattering matrix by  $S(k)$, and by $L$ denote $2N\times2N$ diagonal matrix $\mathrm{Diag}(l_{b})$, see \cite{5} for details.

\begin{lemma}\label{1.3}
Let $L(\mathrm{\Gamma})$ be a Laplace operator on graph $\mathrm{\Gamma}$ with the global scattering matrix, then
$k^{2}\in \mathbb{C}\backslash \{0\}$ is an eigenvalue of $L(\mathrm{\Gamma})$ with the multiplicity $m_{k^{2}}$ if and only if $k$ is a root of the secular equation \eqref{5.3} with the same multiplicity,
 \begin{equation}\label{5.3}
\begin{aligned}
\mathrm{det}(I-S(k)e^{ikL})=0.
\end{aligned}
\end{equation}
\end{lemma}

This lemma was proven in \cite[Theorem 3.7.1]{s} and \cite[Theorem 3.34]{5}. It plays the crucial role in finding the eigenvalues of $L^{st}(\mathrm{\Gamma})$ and $L^{a/st}(\mathrm{\Gamma})$ on inequilateral metric graphs. The following calculation of the eigenvalues of $L^{st}(\mathrm{\Gamma})$ and $L^{a/st}(\mathrm{\Gamma})$ on the special inequilateral non-bipartite graph is based on Lemma \ref{1.3}.

Let graph $\mathrm{\Gamma}$ be a right triangle with edge lengths 3, 4, 5.~We denote the global scattering matrices corresponding to the standard conditions and the anti-standard conditions by $S_{st}$ and $S_{a/st}$, then $S_{st}=-S_{a/st}$.

The secular equation corresponding to the standard conditions is
\begin{center}
  $\mathrm{det}(I-S_{st}e^{ikL})=(e^{12ik}-1)^{2}=0,$
\end{center}
then we have
$$k=\dfrac{\pi}{6},~\dfrac{\pi}{6},~\dfrac{\pi}{3},~\dfrac{\pi}{3},~\dfrac{\pi}{2},~\dfrac{\pi}{2},~\dfrac{2\pi}{3},~\dfrac{2\pi}{3},~\dfrac{5\pi}{6},~\cdots .$$
Since 0 is an eigenvalue of $L^{st}(\mathrm{\Gamma})$ with multiplicity 1, then according to the Lemma \ref{1.3}, the spectrum of $L^{st}(\mathrm{\Gamma})$ is
$$\sigma(L^{st}(\mathrm{\Gamma}))=\{0,(\dfrac{\pi}{6})^{2},(\dfrac{\pi}{6})^{2},(\dfrac{\pi}{3})^{2},(\dfrac{\pi}{3})^{2},(\dfrac{\pi}{2})^{2},(\dfrac{\pi}{2})^{2},(\dfrac{2\pi}{3})^{2}(\dfrac{2\pi}{3})^{2},(\dfrac{5\pi}{6})^{2},\cdots \}.$$
The secular equation corresponding to the anti-standard conditions is
\begin{center}
  $\mathrm{det}(I-S_{a/st}e^{ikL})=(e^{12ik}+1)^{2}=0.$
\end{center}
Through a simple calculation, we can obtain that
$$k=\dfrac{\pi}{12},~\dfrac{\pi}{12},~\dfrac{\pi}{4},~\dfrac{\pi}{4},~\dfrac{5\pi}{12},~\dfrac{5\pi}{12},~\dfrac{7\pi}{12},~\dfrac{7\pi}{12}
,~\dfrac{3\pi}{4},~\cdots,$$
and the spectrum of $L^{a/st}(\mathrm{\Gamma})$ is
$$\sigma(L^{a/st}(\mathrm{\Gamma}))=\{(\dfrac{\pi}{12})^{2},(\dfrac{\pi}{12})^{2},(\dfrac{\pi}{4})^{2},(\dfrac{\pi}{4})^{2},(\dfrac{5\pi}{12})^{2},(\dfrac{5\pi}{12})^{2},(\dfrac{7\pi}{12})^{2},(\dfrac{7\pi}{12})^{2},(\dfrac{3\pi}{4})^{2},\cdots \}.$$

Let $\widetilde{\mathrm{\Gamma}}$ be an equilateral bipartite graph with edge length 1 obtained by adding some vertices on each edge of the metric graph $\mathrm{\Gamma}$, and let $\widetilde{\mathcal{A}}$ and $\widetilde{\mathcal{Z}}$ denote the adjacency matrix and corresponding transition matrix respectively. The eigenvalues of $\widetilde{\mathcal{Z}}$ are given by
$$0,0,-1,1,-\dfrac{1}{2},-\dfrac{1}{2},\dfrac{1}{2},\dfrac{1}{2},-\dfrac{\sqrt{3}}{2},-\dfrac{\sqrt{3}}{2},\dfrac{\sqrt{3}}{2},\dfrac{\sqrt{3}}{2},$$
and the eigenvalues of $L^{st}(\widetilde{\mathrm{\Gamma}})$ and $L^{a/st}(\widetilde{\mathrm{\Gamma}})$~are listed in Table~2,
\begin{center}
\begin{tabular}{c|cccc}
& standard &   anti-standard \\
\hline
   $\lambda=0$ &  $\{0\}$ &         $ \{0\}$ \\
   $ \cos\sqrt{\lambda>0}=1$ &   $\{(2k\pi)^{2}:k\in \mathbb{N}^{2}\}$ &          $\{(2k\pi)^{2}:k\in \mathbb{N}^{2}\}$ \\
    $ \sin\sqrt{\lambda>0}\neq0$ &   $\{(2k\pi\pm\frac{\pi}{6})^{2}:k\in \mathbb{N}_{0}^{2}\}\cup$ &          $\{(2k\pm1)\pi\pm\frac{\pi}{6})^{2}:k\in \mathbb{N}_{0}^{2}\}\cup$\\
    &   $\{(2k\pi\pm\frac{\pi}{3})^{2}:k\in \mathbb{N}_{0}^{2}\}\cup$ &          $\{((2k\pm1)\pi\pm\frac{\pi}{3})^{2}:k\in \mathbb{N}_{0}^{2}\}\cup$\\
     &   $\{(2k\pi\pm\frac{\pi}{2})^{2}:k\in \mathbb{N}_{0}^{2}\}\cup$ &          $\{(2k\pm1)\pi\pm\frac{\pi}{2})^{2}:k\in \mathbb{N}_{0}^{2}\}\cup$\\
     &   $\{(2k\pi\pm\frac{2\pi}{3})^{2}:k\in \mathbb{N}_{0}^{2}\}\cup$ &          $\{(2k\pm1)\pi\pm\frac{2\pi}{3})^{2}:k\in \mathbb{N}_{0}^{2}\}\cup$\\
     &   $\{(2k\pi\pm\frac{5\pi}{6})^{2}:k\in \mathbb{N}_{0}^{2}\}$ &          $\{(2k\pm1)\pi\pm\frac{5\pi}{6})^{2}:k\in \mathbb{N}_{0}^{2}\}$\\
    $ \cos\sqrt{\lambda}=-1$ &    $\{(2k+1)^{2}\pi^{2}:k\in \mathbb{N}_{0}^{2}\}$ &        $ \{(2k+1)^{2}\pi^{2}:k\in \mathbb{N}_{0}^{2}\}$ .\\
  \end{tabular}
  $$\textbf{ Table~2~}$$
\end{center}
Hence the spectra of $L^{st}(\widetilde{\mathrm{\Gamma}})$ and $L^{a/st}(\widetilde{\mathrm{\Gamma}})$~are
$$
  \sigma(L^{st}(\widetilde{\mathrm{\Gamma}}))=\sigma(L^{a/st}(\widetilde{\mathrm{\Gamma}}))=\{0,(\dfrac{\pi}{6})^{2},(\dfrac{\pi}{6})^{2},(\dfrac{\pi}{3})^{2},(\dfrac{\pi}{3})^{2},
  (\dfrac{\pi}{2})^{2},(\dfrac{\pi}{2})^{2},\quad~~$$
$$\quad\qquad\qquad \qquad \qquad\quad(\dfrac{2\pi}{3})^{2},(\dfrac{2\pi}{3})^{2},(\dfrac{5\pi}{6})^{2},
  (\dfrac{5\pi}{6})^{2},\cdots \}.
$$

We have the following results:
 $$\lambda_{k}(L^{st}(\widetilde{\mathrm{\Gamma}}))=\lambda_{k}(L^{a/st}(\widetilde{\mathrm{\Gamma}})),\quad  \lambda_{k}(L^{a/st}(\mathrm{\Gamma}))\leq \lambda_{k+1}(L^{a/st}(\widetilde{\mathrm{\Gamma}})),\quad k\in \mathbb{N}.$$

We can verify that eigenvalues of $L^{st}(\mathrm{\Gamma})$ and $L^{a/st}(\mathrm{\Gamma})$ satisfy
$$\lambda_{k+1}(L^{st}(\mathrm{\Gamma}))\geq \lambda_{k}(L^{a/st}(\mathrm{\Gamma})),\quad k\in \mathbb{N}.$$

This research was supported by the Natural Science Foundation of Hebei Province under Grant No. A2019202205. We thank the reviewers for useful comments and suggestions.

\end{document}